\documentclass[11pt]{article}
\usepackage{geometry}                
\usepackage{graphicx}
\usepackage{amssymb}
\usepackage{epstopdf}
\usepackage{amsthm}
\newtheorem{thm}{Theorem}
\newtheorem{Q}[thm]{Question}
\newtheorem{lemma}[thm]{Lemma}
\newtheorem{corr}[thm]{Corollary}

\newtheorem{res}[thm]{Result}

\usepackage[colorlinks=true, pdfstartview=FitV, linkcolor=blue,
            citecolor=blue, urlcolor=blue]{hyperref}

\title{$s$-Overlap Cycles for Permutations}
\author{Victoria Horan\thanks{\texttt{vhoran@asu.edu}} and Glenn Hurlbert\thanks{\texttt{hurlbert@asu.edu}} \\ School of Mathematics and Statistics \\ Arizona State University \\ Tempe, AZ  85287 USA}

\begin{document}

\maketitle

\begin{abstract}

The goal of this paper is to solve Problem 481 from the list of research problems in the special issue of Discrete Mathematics dedicated to the Banff International Research Station workshop on ``Generalizations of de Bruijn Cycles and Gray Codes" in 2004.  Overlap cycles are generalizations of de Bruijn cycles and Gray codes that were introduced originally in 2010 by Godbole et al.  In this paper we prove that $s$-overlap cycles for $k$-permutations of $[n]$ exist for all $k<n$.

\end{abstract}
	
\section{Introduction}

Many applications require a set of combinatorial objects to be ordered in a specific manner.  One such ordering is an \textbf{$s$-overlap cycle}, or \textbf{$s$-ocycle}.  An $s$-ocycle is an ordering of a set of objects $\mathcal{C}$, each of which has size $n$ and is represented as a string.  The ordering requires that object $b = b_1b_2 \ldots b_n$ follow object $a = a_1a_2 \ldots a_n$ only if $a_{n-s+1}a_{n-s+2} \ldots a_n = b_1b_2 \ldots b_s$.  Ocycles were introduced by Godbole, Knisley, and Norwood in 2010 \cite{Godbole}.

For the reader familiar with universal cycles, we note that an $(n-1)$-ocycle is a universal cycle.  Universal cycles, or ucycles, were introduced in 1992 by Chung, Diaconis, and Graham \cite{UC}, and have been studied extensively ever since, for example over block designs \cite{Dewar}.  Universal cycles for permutations have long been an interesting research problem.  While ucycles using the standard permutation representation are impossible, it has been shown that when $n+1$ symbols are used to represent permutations of $[n]$, ucycles are possible \cite{HurlPerms, JohnsonPerms}.  Another alternative is to consider $k$-permutations, as done by Jackson \cite{JacksonPerms}.  Jackson proved the existence of ucycles for $k$-permutations of $[n]$ with $3 \leq k < n$.  As a natural extension to the problem of finding ucycles for permutations, we consider ocycles.  For example, Figure \ref{Fig1} gives a 3-ocycle for permutations of $[5]$, which is the largest allowable overlap for permutations of $[5]$.  Omitting repeated elements and writing it as a string, the corresponding 3-ocycle is: $$1234512341523 \cdots 34152345.$$

\begin{figure}
$$\begin{array}{l}
12345, 34512, 51234, 23415, 41523, 52314, 31425, 42513, 51324, 32415, 41532, 53214, 21435, \\ 43512, 51243, 24315, 31524, 52413, 41325, 32514, 51423, 42315, 31542, 54213, 21345, 34521, \\ 52134, 13425, 42531, 53124, 12435, 43521, 52143, 14325, 32541, 54123, 12354, 35412, 41235, \\ 23514, 51432, 43215, 21534, 53412, 41253, 25314, 31452, 45213, 21354, 35421, 42135, 13524, \\ 52431, 43125, 12534, 53421, 42153, 15324, 32451, 45132, 13245, 24513, 51342, 34215, 21543, \\ 54312, 31245, 24531, 53142, 14235, 23541, 54132, 13254, 25413, 41352, 35214, 21453, 45312, \\ 31254, 25431, 43152, 15243, 24351, 35124, 12453, 45321, 32145, 14523, 52341, 34125, 12543, \\ 54321, 32154, 15432, 43251, 25134, 13452, 45231, 23145, 14532, 53241, 24153, 15342, 34251, \\ 25143, 14352, 35241, 24135, 13542, 54231, 23154, 15423, 42351, 35142, 14253, 25341, 34152, \\ 15234, 23451, 45123
.	
\end{array}$$
\caption{List of Permutations of $[5]$ in a 3-Ocycle}\label{Fig1}
\end{figure}

In studying ocycles, it is often useful to consider the \textbf{overlap graph} for the set of objects.  This graph has objects represented as vertices.  An edge from object $a_1a_2 \ldots a_n$ to object $b_1b_2 \ldots b_n$ exists if and only if $a_{n-s+1}a_{n-s+2} \ldots a_n = b_1b_2 \ldots b_s$.  In this transition graph, we are looking for a Hamilton cycle, which will correspond to an $s$-ocycle.  In \cite{ResProbs}, the following question was posed by Anant Godbole, which we are able to answer affirmatively in Corollary \ref{main}.

\begin{Q}\label{481}
	\emph{(\cite{ResProbs}, Problem 481.)}  Let $P(k,n,s)$ be the overlap graph for $k$-permutations of $[n] = \{1, 2, \ldots , n\}$.  Is $P(k,n,s)$ hamiltonian whenever $k<n$?
\end{Q}

\section{Results}
	
We begin with a lemma from our previous paper on universal cycles for weak orders.

\begin{lemma}\label{ocmultiset}
	\emph{(\cite{UCWO}, Lemma 4.3)}  Let $n,s \in \mathbb{Z}^+$ with gcd$(n,s)=1$ and $1 \leq s \leq n-2$, and let $M$ be a multiset of size $n$.  Then there is an $s$-ocycle for all permutations of $M$.
\end{lemma}

We can actually improve this and do even better for $s < \frac{n}{2}$.

\begin{lemma}\label{ocmultisetsmall}
	Let $n,s \in \mathbb{Z}^+$ with $1 \leq s < \frac{n}{2}$.  Let $M$ be a multiset of size $n$.  Define the set $A$ to be the set of all permutations of $M$.  Then there is an $s$-ocycle for all permutations of $A$.
\end{lemma}
\begin{proof}
	Construct the transition graph with vertices as $s$-prefixes and $s$-suffixes of words in $A$, and edges representing the words themselves.  Fix an arbitrary vertex $v^{s-}=v_1v_2 \ldots v_s$ as the minimum vertex.  To prove the existence of an Euler tour, and thus prove the existence of an $s$-ocycle, we will show that from any vertex $w^{s-}=w_1w_2 \ldots w_s$, we can find a path to the minimum vertex.

Compare $v^{s-}$ and $w^{s-}$, and consider the first position in which they differ, say index $i$.  In other words, $v_i \neq w_i$, and for all $1 \leq j < i$ we have $v_j = w_j$.  We will choose any string $w=w_1w_2 \ldots w_n \in A$ so that $w$ has $s$-prefix $w^{s-}$.  We have two cases.

\begin{enumerate}
    \item  First, if $v_i \in \{w_{i+1}, w_{i+2}, \ldots , w_s\}$, then we use the following undirected path, in which we merely transpose letters $w_i$ and $v_i$ in $w^{s-}$.
    \begin{eqnarray*}
        w_1w_2 \ldots w_{i-1}w_iw_{i+1} \ldots v_i \ldots w_s & \rightarrow & w_{n-s+1}w_{n-s+2} \ldots w_n \\
        & \leftarrow & w_1w_2 \ldots w_{i-1} v_i w_{i+1} \ldots w_i \ldots w_s
    \end{eqnarray*}
    \item  Second, if $v_i \in \{w_{s+1}, w_{s+2}, \ldots , w_n\}$, then $v_i$ does not appear in vertex $w^{s-}$.  Note that vertex $w^{s-}=w_1w_2 \ldots w_s$ is connected to any $s$-suffix, which consists of all $s$-permutations of the set $$B=M \setminus \{w_1, w_2, \ldots , w_s\}.$$  Thus $v_i \in B$, so we may choose an edge leaving $w^{s-}$ that leads to an $s$-suffix not containing $v_i$, i.e. $v_i$ does not appear in either vertex.  In this case, we may simply replace $w_i$ with $v_i$ as shown in the following path.  
        \begin{eqnarray*}
            w_1w_2 \ldots w_{i-1}w_iw_{i+1} \ldots w_s & \rightarrow & w_{n-s+1}w_{n-s+2} \ldots w_n \\
            & \leftarrow & w_1 w_2 \ldots w_{i-1} v_i w_{i+1} \ldots w_s
        \end{eqnarray*}
\end{enumerate}
At this point, we are one step closer to the minimum vertex, as now the two vertices agree in the first $i$ positions.  Repeating, we will eventually find a path to the minimum vertex.  Thus, the graph is connected.  Finally, since clearly any $s$-suffix of a string in $A$ is also an $s$-prefix, the graph is balanced and hence eulerian.
\end{proof}

The previous lemmas immediately give us the following partial solution when we consider $s$-ocycles on permutations.

\begin{res}\label{betterMSet}
	Let $n,s \in \mathbb{Z}^+$ with $n \geq 2$.  If either (1) $1 \leq s < \frac{n}{2}$, or (2) gcd$(s,n)=1$ with $\frac{n}{2} \leq s <n-1$, then there exists an $s$-ocycle on the set of permutations of $[n]$.
\end{res}
\begin{proof}
	This is merely a concise restatement of Lemmas \ref{ocmultiset} and \ref{ocmultisetsmall} with $M = \{1, 2, \ldots , n\}$.
\end{proof}	

This result gives us an alternative modification for dealing with the problem of finding ucycles for permutations.  Instead of increasing the alphabet size from $n$ to $n+1$ as done in \cite{HurlPerms, JohnsonPerms}, an alternative could be to consider the largest possible overlap, with $s=n-2$ being the best alternative to a ucycle.  Note that this result also shows that under the given conditions, all permutations of an $[n]$-set are connected within a larger transition graph.  We can extend the ocycle result to $k$-permutations of $[n]$ to partially answer Question \ref{481}.

\begin{res}\label{kpermOC}
	Let $n,s,k \in \mathbb{Z}^+$ with $1\leq k<n$.  If either (1) $1 \leq s < \frac{k}{2}$, or (2) gcd$(s,k)=1$ with $\frac{k}{2} \leq s < k-1$, then there exists an $s$-ocycle on the set of $k$-permutations of $[n]$.
\end{res}
\begin{proof}
	First, construct the transition graph with vertices of length $s$ ($s$-prefixes of $k$-permutations) and edges representing $k$-permutations.  We allow an edge from vertex $u$ to vertex $v$ if and only if $u$ is an $s$-prefix and $v$ is an $s$-suffix for some $k$-permutation.  We will show that this graph is balanced and weakly connected, and thus is eulerian.  From an Euler tour, we can find an $s$-ocycle.
	
	Define the minimum $k$-permutation $v = 123 \ldots k$, and the minimum vertex, $v^{s-}$, in the transition graph to be the $s$-prefix of $v$.  Let $w^{s-} = w_1 w_2 \ldots w_s$ be the prefix of an arbitrary $k$-permutation $w_1 w_2 \ldots w_k$.  By Result \ref{betterMSet}, all permutations of a $k$-set are connected under our hypthoses, so we may assume that $w$ is ordered as $w_1 < w_2 < \cdots < w_k$.  We will show that there is a path in the graph from $w^{s-}$ to $v^{s-}$.  Define $D = \{w_1, w_2, \ldots , w_k\} \setminus \{1,2,3, \ldots , k\}$ and $\overline{D} = \{1, 2, 3, \ldots , k\} \setminus \{w_1, w_2, \ldots , w_k\}$.
	
	Note that if $D = \emptyset$, then $w$ is a permutation of $v$ and so by the comments following Result \ref{betterMSet}, we know that there exists a path.  For $D \neq \emptyset$, we choose $d$ letters $a_1, a_2, \ldots , a_d \in \overline{D}$, where $d = \min\{k-s,|\overline{D}|\}$.  If $d < k-s$, then we may also select letters $a_{d+1}, a_{d+2}, \ldots , a_{k-s} \in \{w_{s+1}, w_{s+2}, \ldots , w_k\}$.  Now in our graph we follow the edge corresponding to the $k$-permutation $w_1w_2 \ldots w_s a_1 a_2 \ldots a_{k-s}$.  By following this edge, we have found a $k$-permutation with more letters in common with $v$ than $w$ did.  Since Result \ref{betterMSet} implies that all permutations of a $k$-set are connected in the transition graph, we may arrange this string in increasing order, and by repeating this procedure we will eventually find a $k$-permutation that is simply a permutation of $v$, at which point we are done.
	
	Since we have shown that the graph is connected, we need only show that the graph is balanced in order to prove that an Euler tour exists.  However it is clear that the graph is balanced, as any prefix of a $k$-permutation is also a suffix of a $k$-permutation.
\end{proof}

We now prove our main result, which will provide a complete solution to Question \ref{481}.

\begin{res}\label{allkpermOC}
	For all $n,s,k \in \mathbb{Z}^+$ with $1 \leq s < k < n$, there is an $s$-ocycle for $k$-permutations of $[n]$.
\end{res}
\begin{proof}
	We construct the standard transition graph $G$ where vertices of length $s$ correspond to $s$-prefixes of $k$-permutations of $[n]$, and edges correspond to $k$-permutations of $[n]$.  If we can show that this graph is eulerian, then we have shown that there exists an $s$-ocycle for $k$-permutations of $[n]$.  First, we note that since any $s$-prefix of a $k$-permutation is also an $s$-suffix of a $k$-permutation, the graph is balanced.  All that remains is to show that the graph is connected.
	
	Define the minimum vertex $v^{s-} = 12 \ldots s$, and let $w^{s-} = w_1w_2 \ldots w_s$ be an arbitrary vertex in the graph, which we assume to be an $s$-prefix of the $k$-permutation $w = w_1 w_2 \ldots w_k$.  We will frequently refer to \textbf{rotations} of a vertex.  This is defined as following the edges of the cycle corresponding to rotations of a $k$-permutation that the vertex is an $s$-prefix for.
	
	We next compare $w^{s-}$ with $v^{s-}$.  Let $i$ be the first index in which $w_i \neq v_i$ and let $g = \hbox{gcd}(s,k)$.  If $g=1$, we are done by Result \ref{kpermOC}.  Otherwise, we note that rotations of $w$ partition the string into blocks of length $g$.  All addition in indices will be modulo $k$.  We have two cases.
	\begin{description}
		\item[Case 1:]  If $i \not \in \{w_{i+1}, w_{i+2}, \ldots , w_k\}$:
		
			Rotate $w$ so that $w_i$ is in the first block.  This means that we are considering some vertex $$w_{i-t} w_{i-t+1} \ldots w_{i-t+s-1}$$ with $t \leq g-1$, or $i \in [i-t,i-t+g-1]$.  Follow the edge out of this vertex that corresponds to a rotation of $w$.  This takes us to the vertex $$w_{i-t+k-s} w_{i-t+k-s+1} \ldots w_{i-t+k-1}.$$  Next we follow the backwards edge corresponding to the $k$-permutation $$w_{i-t}w_{i-t+1} \ldots w_{i-1} (i) w_{i+1} \ldots w_{i-t+k-1}.$$  Now we are at the vertex $$w_{i-t}w_{i-t+1} \ldots w_{i-1} (i) w_{i+1} \ldots w_{i-t+s-1}.$$  Finally we follow rotations of this vertex to end at the vertex $$12 \ldots (i) w_{i+1} w_{i+2} \ldots s.$$  This vertex is closer to the minimum vertex since the first $i$ letters agree.  Repeating this procedure, we will eventually arrive at the minimum vertex.
		\item[Case 2:]  If $i \in \{w_{i+1}, w_{i+2}, \ldots , w_k\}$:
		
		Rotate $w$ so that $i$ is in the first block.  We are now considering some vertex $$a_1a_2 \ldots a_s$$ with $i \in \{a_1, a_2, \ldots , a_g\}$, so assume that $i = a_j$.  Note that $$|[n] \setminus \{w_1, w_2, \ldots , w_k\}| > 1,$$ so choose some $x$ in this set.  We follow the edge from $a_1 a_2 \ldots a_s$ corresponding to a rotation of $w$ $$a_1 a_2 \ldots a_{j-1} (i) a_{j+1} \ldots a_k.$$  This takes us to the vertex $$a_{k-s+1} a_{k-s+2} \ldots a_k,$$ which does not contain the letter $i$.  From this vertex we follow the backward edge $$a_1 a_2 \ldots a_{j-1} (x) a_{j+1} \ldots a_k.$$  Note that $i$ does not appear in this edge, so we can go to Case (1).
	\end{description}
	
	When we have finished, we will have arrived at the minimum vertex.  Thus the graph is weakly connected, and so is eulerian.
\end{proof}

Finally, the following answer to Question \ref{481} is an obvious Corollary to Result \ref{allkpermOC}.

\begin{corr}\label{main}
	The permutation overlap graph $P(k,n,s)$ is hamiltonian whenever $k<n$.
\end{corr}

\end{document}